\documentclass[12pt]{article}
\usepackage{amsmath,amsthm,amssymb}
\usepackage{graphicx,hyperref}

\newcommand{\A}{\mathcal{A}}
\newcommand{\poly}{\mathcal{P}}
\def\P{\mathbb{P}}
\def\E{\mathbb{E}}

\DeclareMathOperator{\Bin}{Bin}

\newtheorem{theorem}{Theorem}[section]

\newtheorem{lemma}[theorem]{Lemma}

\newtheorem{corollary}[theorem]{Corollary}

\newcommand{\AND}{\textsf{AND}}

\newtheorem{lem}[theorem]{Lemma}
\newtheorem{prop}[theorem]{Proposition}
\newtheorem{clm}[theorem]{Claim}
\newtheorem{claim}[theorem]{Claim}

\newtheorem*{thm*}{Theorem}

\theoremstyle{definition}

\theoremstyle{remark}

\numberwithin{equation}{section}

\newcommand{\abs}[1]{\left\vert#1\right\vert}

\newcommand{\set}[1]{\left\{#1\right\}}

\newcommand{\br}[1]{\left[#1\right]}

\newcommand{\sr}[1]{\left(#1\right)}

\newcommand{\N}{\mathbb{N}}
\newcommand{\Q}{\mathbb{Q}}
\newcommand{\R}{\mathbb{R}}

\newcommand{\eqdef}{\stackrel{\mathrm{def}}{=}}

\renewcommand{\Pr}{}
\let\Pr\relax
\DeclareMathOperator*{\Pr}{\mathbb{P}}

\def\squareforqed{\hbox{\rlap{$\sqcap$}$\sqcup$}}
\def\qed{\ifmmode\squareforqed\else{\unskip\nobreak\hfil
\penalty50\hskip1em\null\nobreak\hfil\squareforqed
\parfillskip=0pt\finalhyphendemerits=0\endgraf}\fi}

\newcommand{\vphi}{\varphi}

\newcommand{\Ze}{\mathrm{Zeros}}

\newcommand{\sett}[2]{ \left\{  #1 \ : \ #2  \right\} }

\newcommand{\Choose}[2]{\left(#1 \atop #2 \right)}

\newcommand{\intn}{I_n}

\begin{document}

%\author{Ron Peled\thanks{\texttt{peledron@stat.berkeley.edu.} Department of Statistics, UC Berkeley.
%Supported by Microsoft Research and NSF grant DMS-0605166. } \and
\author{Ron Peled\thanks{\texttt{peled@cims.nyu.edu} Courant Institute, New York University.
Research supported by Microsoft Research and NSF grant DMS-0605166. } \and
Ariel Yadin\thanks{\texttt{A.Yadin@statslab.cam.ac.uk.} Center for Mathematical Sciences, University of Cambridge, UK.} \and Amir
Yehudayoff\thanks{\texttt{amir.yehudayoff@gmail.com.} Institute for Advanced Study, Princeton NJ.
Partially supported by NSF grant CCF 0832797.  }}
\title{The Maximal Probability that $k$-wise Independent Bits are All $1$}

\date{}
\maketitle

\begin{abstract}
A $k$-wise independent distribution on $n$ bits is a joint
distribution of the bits such that each $k$ of them are
independent. In this paper we consider $k$-wise independent
distributions with identical marginals, each bit has probability
$p$ to be $1$. We address the following question: how high can the
probability that all the bits are 1 be, for such a distribution?
For a wide range of the parameters $n,k$ and $p$ we find an
explicit lower bound for this probability which matches an upper
bound given by Benjamini et al., up to multiplicative factors of lower order. 
%ref
In particular, for fixed $k$, we obtain the sharp asymptotic behavior.
%end ref
The question we investigate can be viewed as a relaxation of a major
open problem in error-correcting codes theory, namely, how large
can a linear error correcting code with given parameters be?

The question is a type of discrete moment problem, and our
approach is based on showing that bounds obtained from the theory
of the classical moment problem provide good approximations for
it. The main tool we use is a bound controlling the change in the
expectation of a polynomial after small perturbation of its zeros.
\end{abstract}
\begin{section}{Introduction}
The problem of generalized inclusion-exclusion inequalities has
been considered by many authors
\cite{B1854,B37,DS67,K75,P88,BP89,GX90,LN90}. In this problem one
has $n$ events $A_1,\ldots, A_n$ and the probabilities of
intersections $\cap_{i\in S} A_i$ for all $S$ with $|S|\le k$.
Given this information the goal is to bound the probability of
$\cup_{i=1}^n A_i$ from above and from below. The classical
Bonferroni inequalities state that the odd and even partial sums
of the inclusion-exclusion formula provide such upper and lower
bounds, respectively. But in many cases these bounds are far from
being sharp, in the sense that much tighter bounds may be deduced
from the same information.

In this paper we address a special case of this question. In our
setting the events all have equal probability $\P(A_i)=p$ and are
$k$-wise independent; that is, $\P(\cap_{i\in S} A_i)=p^{|S|}$
whenever $|S|\le k$. When referring to this case we shall use a
slightly different terminology and refer to the events
$A_1,\ldots, A_n$ as $n$ bits. 
%ref
For convenience, we consider the intersection of events instead of the union,
which is equivalent by de Morgan's rules.
%end ref
With this terminology we are
interested in estimating the probability of the $\AND$ of the bits
given that their joint distribution is $k$-wise independent with
identical marginals $p$. Besides the simplification arising
from considering a particular case, this case is of special
interest from several points of view.

First, $k$-wise independent distributions play a key role in the
computer science literature where they are used for
derandomization (there are many references, e.g. the survey
\cite{LW95}). Here is an example for the use of $k$-wise
independence in this context: Assume that a given efficient
\emph{probabilistic} algorithm $A$ works, even when the algorithm uses
pairwise independent bits instead of truly independent random bits. Since
there are pairwise independent distributions with small support,
this implies that the algorithm can be converted to an \emph{efficient
deterministic} algorithm. In order to prove that $A$ indeed works
with access to only pairwise independent bits, one needs to show that the
probabilities of certain events (that depend on $A$)
do not change significantly when ``moving'' to a pairwise independent distribution.

Second, there is a strong connection between linear error
correcting codes and $k$-wise independent distributions (when
$p=\frac{1}{q}$ for a prime power $q$). Given a linear
error-correcting code over $(GF(q))^n$ with minimal distance $d$,
one may obtain a $k$-wise independent distribution with $k=d-1$ and
$p=\frac{1}{q}$ by 
%ref - replaced "considering" by:
sampling uniformly at random from  
%end ref
the dual of the code and replacing the resulting codeword by the indicator word of its zeros. Although by this construction one gets only distributions with a certain structure,
this is by far the most common way to construct $k$-wise independent distributions.
%NEW
%According to this construction of $k$-wise independent distributions,there is 
It gives a simple connection between the size of the code $C$,
and the probability of getting the all $1$'s vector:
$$\P[(1,\ldots,1)] = \frac{1}{|C^\perp|} = \frac{|C|}{q^n},$$
where the probability is over the $k$-wise independent distribution constructed
from $C$, and $C^\perp$ is the dual of $C$.
%END NEW
A very basic and open question in the theory of
error correcting codes is how large can a linear error-correcting
code be, for given $n,d,q$ (\cite{MS77}, see also \cite{DY04}).
A large code immediately implies a large probability for the $\AND$ of the bits,
hence investigating the maximal probability that the $\AND$ event can achieve for a given
triplet $n,k,p$ can be thought of as a relaxation of the error correcting codes question.
%NEW
However, in general, these two questions turn out not to be equivalent,
even asymptotically in $n$, as an example from \cite{BGGP} shows:
\begin{description}
\item[(i)] For every $3$-wise independent distribution $\mu$ on $n$ bits with marginal
probabilities $1/3$ that is obtained from a linear code as (roughly) described above,
$\mu[(1,\ldots,1)] = O(\frac{1}{n \log n})$
(this is a version of Roth's theorem on 3-term arithmetic progressions for $(GF(3))^n$, see \cite{M95}.)
\item[(ii)] There exists a $3$-wise independent distribution $\mu'$ on $n$ bits with marginal
probabilities $1/3$ such that $\mu'[(1,\ldots,1)] = \Omega(\frac{1}{n})$.
\end{description}
%END NEW
%ref
An important property of the code-based constructions of $k$-wise independent
distributions is that such distributions have small support.
The support size is important for derandomization, as discussed above.
In this paper we show existence of $k$-wise independent distributions that
assign large probability to $(1,\ldots,1)$,
but we do not show that they have small support.
%end ref

Third, the question has intrinsic mathematical beauty. From an analytic perspective, when
attempting its solution one is naturally led to discrete analogues
of classical moment problems (classical quadrature formulas).
Although some investigation of such discrete moment problems
exists in the literature \cite[Chap. VIII]{KN77},\cite{P88,BP89}, they are much
less understood than their classical counterparts.
Still, the classical theory sheds light on our problem and enables us to make
progress on it and obtain quite precise answers. From a more geometric standpoint, the
set of $k$-wise independent distributions is an interesting convex body, the structure
of which we understand quite poorly. In this work we try to at least understand the projection
of this body in one specific direction.

Finally, in the case $p=\frac{1}{2}$, the maximal probability of the $\AND$ event is also
the maximal probability for any fixed string of bits
(roughly, `translating' a distribution by a constant vector, does not `affect' the $k$-wise independence).
In other words, for $p=\frac{1}{2}$ this maximal probability corresponds to the minimal
 min-entropy possible for a $k$-wise independent distribution, which seems a very basic property.

This work continues a previous work \cite{BGGP}
in which an (explicit) upper bound for the $\AND$ event was found
(as well as some lower bounds). The upper bound was derived as a solution to a relaxed maximization problem (see Section~\ref{relaxed_problem_section}) which appears quite similar to the original problem. The similarity makes it natural to expect that the upper bound be quite close to the true maximal probability. Indeed, in this work we affirm this expectation in a large regime of the parameters. 
%ref - removed "To describe our results let us first introduce some notation."

\subsection{Results}

Denote by $M(n,k,p)$ the maximal probability of the $\AND$ event
for a $k$-wise independent distribution on $n$ bits with marginals
$p$. For odd $k$ it is shown in \cite{BGGP} that
\begin{equation}\label{odd_case_equality}
M(n,k,p) = pM(n-1,k-1,p) \qquad \text{($k$ odd)},
\end{equation}
hence it is enough to consider the case of even $k$. It is also shown there that
\begin{equation}\label{tilde_M_bounds_M}
M(n,k,p)\le\tilde{M}(n,k,p)
\end{equation}
where $\tilde{M}(n,k,p)$ is the solution to a certain maximization problem (see Section~\ref{relaxed_problem_section}) and satisfies for even $k$,
\begin{equation}\label{value_of_tilde_M}
\tilde{M}(n,k,p) = \frac{p^n}{\P(\Bin(n,1-p)\le \frac{k}{2})}.
\end{equation}

%NEW   (actually changed)
%For even $k$, denote by
%$$ \tilde{M}(n,k,p) = \frac{p^n}{\P(\Bin(n,1-p)\le \frac{k}{2})}  $$
%(this is the upper bound from \cite{BGGP}).
%\begin{maintheorem} \label{thm: main}
%There exist constants $c_1,c_2,c_3 > 0$ such that the following holds.
%Let $n \in \N$, $k \in \N$ even, and $0 < p < 1$.
%Let $N = np(1-p) -1$.
%
%If $k \leq c_1 \cdot N$, then
%\begin{equation*}
%\frac{c_3}{k} \exp \sr{- \frac{c_2 \cdot k}{V(N/k)}} \tilde{M}(n,k,p)\le M(n,k,p)\le \tilde{M}(n,k,p),
%\end{equation*}
%where $V(a)=\exp \Big( \sqrt{\log(a)\log\log(a)} \Big)$.
%\end{maintheorem}
%
%%NEW
%The theorem is stated for even $k$;
%the case where $k$ is odd
%can be analyzed using \eqref{odd_case_equality} below.
%END NEW

%In this section we present our main result. Recall that $M(n,k,p)$
%stands for the maximal probability of the $\AND$ event under a
%$k$-wise independent distribution on $n$ bits with marginal
%probability $p$. Recall also that $\tilde{M}(n,k,p)$ is given in
%\eqref{value_of_tilde_M} and that $M(n,k,p)\le\tilde{M}(n,k,p)$.
Our main result is a lower bound for $M(n,k,p)$ matching the bound
given by $\tilde{M}(n,k,p)$ up to multiplicative factors of lower
order, in a large regime of the parameters. Specifically:
\begin{theorem} \label{main_theorem}
There exist constants $c_1,c_2,c_3 > 0$ such that the following holds.
Let $n \in \N$, $k \in \N$ even, and $0 < p < 1$.
Let $N = np(1-p)-1$.
Assume
\begin{equation} \label{eqn:  constraint on k}
k \leq c_1 \cdot N.
\end{equation}
Then,
\begin{equation}
M(n,k,p)\ge \frac{c_3}{k}
\exp \sr{-c_2 \cdot \frac{k}{V(N/k)}} \tilde{M}(n,k,p),
\end{equation}
where $V(a)=\exp \Big( \sqrt{\log(a)\log\log(a)} \Big)$.
\end{theorem}

The cases where \eqref{eqn:  constraint on k} does not hold
are not covered by Theorem \ref{main_theorem}.
Some partial results on these cases were given in
\cite{BGGP}. For the case $n(1-p)\le\frac{k}{2}$ the bound $p^n\le
M(n,k,p)\le \tilde{M}(n,k,p)\le 2p^n$ was shown, and for the case
$(n-1)p\le1$ it was shown that $M(n,k,p)=p^k$. The case $k=2$ was
also solved there.

To better understand the bound given in Theorem~\ref{main_theorem}, 
we present some particular cases in the following
\begin{corollary}\label{main_result_corollary}
There exist $C,c>0$ such that for all $n \in \N$, $k \in \N$ even, and $0 < p < 1$, letting $N = np(1-p)-1$ we have
%for all $n,k,p$ as in Theorem \ref{main_theorem}. Let $N=np(1-p)-1$.
\begin{enumerate}
\item For every $m>0$, there exists $N_0 = N_0(m)$ such that
if $N > N_0$ and $k \leq (\log N)^m$, then
\begin{equation*}
M(n,k,p) \geq \frac{c}{k}\tilde{M}(n,k,p) .
\end{equation*}

\item For every $0 < \beta < 1$, there exists $c(\beta) > 0$ and $N_0=N_0(\beta)$ such that
if $N>N_0$ and $k \leq N^{\beta}$, then
\begin{equation*}
M(n,k,p) \geq c \exp \sr{-\frac{k}{\exp (c(\beta) \sqrt{\log (k)\log\log(2k)})} } \tilde{M}(n,k,p) .
\end{equation*}
\item For any $k$ satisfying $k\le cN$,
\begin{equation*}
M(n,k,p)\ge ce^{-C k}\tilde{M}(n,k,p) .
\end{equation*}
\end{enumerate}
\end{corollary}

%To understand Theorem~\ref{thm: main} better, we highlight some of its special cases.
By estimating $\tilde{M}(n,k,p)$ (using \eqref{value_of_tilde_M} and Claim~\ref{tilde_M_estimates} below) in the first two cases of the above corollary we obtain, using \eqref{tilde_M_bounds_M}, explicit two-sided bounds on $M(n,k,p)$. They show that for a large range of the parameters, the leading order behavior of $M(n,k,p)$ is identified and for the case of constant $k$, the exact asymptotics is determined, as follows:
%The bounds of the previous corollary translate, using \eqref{tilde_M_bounds_M} and \eqref{value_of_tilde_M}, to the following explicit two-sided bounds on $M(n,k,p)$.
%Plugging in the value of $\tilde{M}$ from \eqref{value_of_tilde_M} and \eqref{tilde_M_bounds_M}
\begin{corollary}\label{intro_corollary}
There exist $C,c>0$ such that for all $n \in \N$, $k \in \N$ even, and $0 < p < 1$, letting $N = np(1-p)-1$ we have
%There exists $C>0$ such that if $k\ge 2$ is even and $N=np(1-p)-1$ then
%There exists $C>0$ such that for all even $k\ge 2$:
\begin{enumerate}
\item For every $m>0$, there exists $N_0 = N_0(m)$ such that
if $N > N_0$ and $k \leq (\log N)^m$, then
%If $k\le (\log N)^m$ for some $m>0$ and $N>N_0(m)$, then
\begin{equation}
\frac{c}{\sqrt{k}}\left(\frac{pk}{2e(1-p)n}\right)^{k/2}\le M(n,k,p)\le C\sqrt{k}\left(\frac{pk}{2e(1-p)n}\right)^{k/2}.
\end{equation}
\item For every $0 < \beta < 1$, there exists $c(\beta) > 0$ and $N_0=N_0(\beta)$ such that
if $N>N_0$ and $k \leq N^{\beta}$, then
%If $k\le N^{\beta}$ for some $0<\beta<1$ and $N>N_0(\beta)$, then
\begin{equation}
\frac{1}{C}e^{-\frac{k}{U(k,\beta)}}\left(\frac{pk}{2e(1-p)n}\right)^{k/2}\le M(n,k,p)\le C\sqrt{k}e^{\frac{k^2}{2n}}\left(\frac{pk}{2e(1-p)n}\right)^{k/2},
\end{equation}
where $U(k,\beta) \eqdef \exp (c(\beta) \sqrt{\log (k)\log\log(2k)})$.% and $c(\beta)>0$ is a constant.
\end{enumerate}
\end{corollary}

Let us compare this with known results, our novelty is in the lower bounds and
so we only compare these. As far as the authors are aware, the best known lower bounds
for $M(n,k,p)$ come from error-correcting codes and apply to the cases when
$p=\frac{1}{q}$ for a prime power $q$. The most important case for applications is $p=\frac{1}{2}$.
In this case it was known using BCH codes (\cite{MS77},\cite[Chapter 15]{AS00})
that $M(n,k,\frac{1}{2})\ge \left(\frac{c_1}{n}\right)^{\lfloor k/2 \rfloor}$ and also using the
Gilbert-Varshamov bound \cite{MS77} that $M(n,k,\frac{1}{2})\ge c_2\left(\frac{c_3(k-1)}{n}\right)^{k-1}$
for some constants $c_1,c_2,c_3>0$. In both cases our bound improves on the known asymptotic results
for $k=o(n)$, but still growing to infinity with $n$.

Other cases where lower bounds were known are the cases in which
$p=\frac{1}{q}\neq \frac{1}{2}$ for a prime power $q$. In these cases much less is known
and even for the case of constant $k$ and $p$, the best results we are aware of are
of the form $M(n,k,p)\ge n^{-\alpha(k,p)(1+o(1))}$
where, except for a few cases,
$\alpha(k,p)$ is strictly larger than $\lfloor\frac{k}{2}\rfloor$ (see \cite{DY04}
for a survey of such results). For example, in the case $p=\frac{1}{3}$ and constant $k\ge 7$ it appears that the best known
asymptotic result in $n$ was $M(n,k,\frac{1}{3})\ge cn^{-\lceil 2(k-1)/3\rceil}$. Our results show that the correct asymptotic behavior
for constant $k$ and $p$ is $M(n,k,p)=\Theta(n^{-\lfloor k/2\rfloor})$.% (the case of odd $k$
%follows from the case of even $k$, see \eqref{odd_case_equality} below).

Here is a high-level description of the proof of Theorem~\ref{main_theorem}.
We start by employing linear programming duality as in \cite{BGGP}.
This duality shows that $M(n,k,p)$ is the minimum of the expectation $\E f(X)$,
where $X\sim\Bin(n,p)$, over all polynomials $f$ from a certain class (see \eqref{one_d_max_dual_discrete_problem}). A similar duality shows that $\tilde{M}(n,k,p)$ is the minimum of the expectation $\E g(X)$, where $X\sim\Bin(n,p)$, over all polynomials $g$ from a strictly smaller class than that of the first minimization problem (see \eqref{one_d_max_dual_continuous_problem}). This latter minimization problem is exactly solvable using the methods of the classical moment problem. 
%Following \cite{BGGP}, we note that $\tilde{M}(n,k,p)$ is exactly computable using the methods of the classical moment problem. 
We continue by associating to each polynomial $f$ from the class of the first problem, a polynomial $g$ from the class of the second problem, obtained by perturbing the roots of $f$. 
%We continue by estimating $\E f(X)$ using $\E g(X)$,
%where $g$ is polynomial obtained by perturbing the roots of $f$ in a certain way.
It thus follows that
\begin{equation*}
\frac{\tilde{M}(n,k,p)}{M(n,k,p)}\le \max \left(\frac {\E g(X)}{\E f(X)}\right)
\end{equation*}
where the maximum ranges over all polynomials $f$ from the class of the first problem and $g$ is the polynomial associated to $f$.
%We thus wish to bound the ratio $\E g(X) / \E f(X)$.
A bound for the RHS of the above inequality which yields Theorem~\ref{main_theorem} is then given by Theorem~\ref{thm: ratio between
polynomials}.
Our methods can be used to bound the `change' in expectation for
other distributions as well (see Section~\ref{sec: pert root} for
more details). 
Such an argument
%We note that the methods used in proving the theorem 
can be applied to other problems where there is a classical moment problem analogue to
discrete problems. It thus seems that Theorem~\ref{thm:  ratio
between polynomials} and its proof might be of independent interest.

\paragraph{Outline}
Section~\ref{problem_and_dual_section} gives a more precise description of the question we consider, 
and explains some useful facts about it,
including the use of linear programming duality.
Section~\ref{relaxed_problem_section} describes the relaxed version of the problem with emphasis on its similarity to the original problem. Our main result is explained in Section~\ref{main_results_section} where the result on polynomials and the reduction between them are described. We also do the computations needed to obtain Corollary~\ref{intro_corollary} there. Finally, Section~\ref{sec: pert root} proves the result on polynomials. Some open problems are presented in Section~\ref{open_problems_section}. For completeness, the appendix gives short proofs for the results of \cite{BGGP} that we use.
\end{section}

\begin{section}{The problem and its dual} \label{problem_and_dual_section}
In this section we introduce notation for our problem and present
it in more precise terms. We then continue to describe the dual of
the problem, on which we shall concentrate in the following
sections. Let $\A(n,k,p)$ be the set of all probability
distributions on $\{0,1\}^n$ which are $k$-wise independent and
have identical marginals $p$. In other words, the distribution of
$(X_1,\ldots, X_n)$ belongs to $\A(n,k,p)$ if $\P(\forall \ i\in S
\ X_i=1)=p^{|S|}$ for all $S$ with $|S|\le k$.
Thinking of $\A(n,k,p)$ as a body in $\R^{2^n}$,
it is convex. Hence, bounding the probability of the event $\AND = \{
\forall \ 1\le i\le n \ X_i=1 \}$ under all probability
distributions in $\A(n,k,p)$ is the same as finding
\begin{align}
M(n,k,p) &= \max_{\Q\in\A(n,k,p)} \Q(\AND) \label{one_d_non_symmetric_max_problem}\\
m(n,k,p) &= \min_{\Q\in\A(n,k,p)} \Q(\AND) \label{one_d_non_symmetric_min_problem}.
\end{align}

In \cite{BGGP} it was shown that for many choices
of the parameters $n,k,p$ we have $m(n,k,p)=0$, making the bound
in this direction perhaps less useful. In this work we concentrate
on estimating $M$.
%For odd $k$ it is shown in \cite{BGGP} that
%\begin{equation}\label{odd_case_equality}
%M(n,k,p) = pM(n-1,k-1,p) \qquad \text{($k$ odd)},
%\end{equation}
%hence it is enough to consider the even case.

A simplification of problems \eqref{one_d_non_symmetric_max_problem} and \eqref{one_d_non_symmetric_min_problem} is possible:
Define the set $$ \A^s(n,k,p) \subseteq \A(n,k,p)$$ to be the set of symmetric distributions
in $\A(n,k,p)$; that is, the joint distribution of $(X_1,\ldots, X_n)$ is in $\A^s(n,k,p)$
if it is in $\A(n,k,p)$ and $(X_1,\ldots, X_n)$ are exchangeable.
Since the $\AND$ event is
symmetric, one can show that
\begin{align}
M(n,k,p)&=\max_{\Q\in\A^s(n,k,p)} \Q(\AND) \label{one_d_max_discrete_problem}\\
m(n,k,p)&=\min_{\Q\in\A^s(n,k,p)} \Q(\AND) \label{one_d_min_discrete_problem}.
\end{align}
Note further that a distribution in $\A^s(n,k,p)$ may be identified with the integer random
variable $S$ which counts the number of bits that are $1$.
 Note that such an
 $S$ has the following properties:
\begin{enumerate}
\item[(I)] $S$ is supported on $\{0,1,\ldots, n\}$.
\item[(II)] $\E S^i = \E X^i$ for $X\sim \Bin(n,p)$ and $1\le i\le k$.
\end{enumerate}
The converse also holds (see \cite{BGGP}); that is,
\begin{lemma}
\label{lem: S and As are same}
For each random variable $S$ satisfying $(I)$ and $(II)$,
there exists $\Q\in\A^s(n,k,p)$ such that $S$ has the distribution of the number of bits which are $1$ under $\Q$.
\end{lemma}
Relying on Lemma~\ref{lem: S and As are same}, we shall henceforth identify $\A^s(n,k,p)$
with distributions $S$ satisfying $(I)$ and $(II)$ above. There is a short argument
given below showing that the distribution $S$ achieving the maximum in \eqref{one_d_max_discrete_problem}
is unique. Similar arguments are used in \cite{KN77}.

We can now think of problem $\eqref{one_d_max_discrete_problem}$ as a linear programming problem in $n+1$ variables, namely, find the maximum of $\P(S=n)$ under the constraints $\P(S=i)\ge0$ for $0\le i\le n$, $\sum_{i=0}^n \P(S=i)=1$ and the linear conditions on
$\P(S=i)$ given by (II) above.
We shall estimate $M(n,k,p)$ using
the dual linear programming problem \cite{BGGP}:
\begin{equation} \label{one_d_max_dual_discrete_problem}
M(n,k,p) = \min_{P\in \poly_k^d} \E_{\Bin(n,p)} P(X) ,
\end{equation}
where $\poly_k^d$ is the collection of polynomials $P:\R\to\R$
of degree at most $k$
satisfying $P(i)\ge 0$ for $i\in\{0,1,\ldots,n-1\}$ and $P(n)\ge 1$ (the $d$ in the notation stands for discrete).
 We shall bound
$M(n,k,p)$ from below by showing that for each $P\in\poly_k^d$, the above expectation is not too small.

Note that finding an optimal polynomial for the above problem gives more information
than just $M(n,k,p)$. By the theorem of complementary slackness of linear programming, if $Z$ is the set of zeros of an optimal polynomial in \eqref{one_d_max_dual_discrete_problem} then the support of the optimal distribution in \eqref{one_d_max_discrete_problem} is contained in $Z\cup\{n\}$.
This can also be seen
probabilistically since if $P$ is an optimal polynomial,
then $\E P(S)=M(n,k,p)$ for any $S\in\A^s(n,k,p)$
(since $P$ is of degree at most $k$). But for any $P\in\poly_k^d$ we have $\E P(S)\ge \P(S=n)$,
hence $\P(S=n)=M(n,k,p)$ only when $P(n)=1$ and all the support of $S$ besides $\{n\}$ is
contained in the zero set of $P$. Of course once the support of the optimal $S$
(or the zero set of an optimal polynomial) is known, the exact probabilities of $S$ can be found by solving a system of linear equations.
This system always has a unique solution (it has a Van der Monde coefficient matrix), which also proves the uniqueness of the distribution of $S$.

Pr\'ekopa in his work (\cite{P88}, see also \cite{BP89})
considers in more generality the problem of estimating $\P(S=n)$ for the class of random variables
$S$ with given first $k$ moments (not necessarily those of the Binomial).
He does not use probabilistic language and instead writes his work in linear programming terminology. Adapting one of his results to our situation,
it reads
\begin{theorem} \label{prekopa_theorem}(Pr\'ekopa \cite[Theorem 9]{P88})
There exists an optimizing polynomial $P$ for \eqref{one_d_max_dual_discrete_problem} of the following form. $P(n)=1$ and $P$ has $k$ simple roots $z_1<z_2<\cdots<z_k$, all contained in $\{0,1,\ldots, n-1\}$. Furthermore
\begin{enumerate}
\item For even $k$, the roots come in pairs $z_{i+1} = z_i + 1$ for odd $1\le i\le k-1$.
\item For odd $k$, $z_1=0$ and the rest of the roots come in pairs $z_{i+1} = z_i + 1$ for even $2\le i\le k-1$.
\end{enumerate}
\end{theorem}
This result is also essentially contained in \cite[Chap. VIII, sec. 3]{KN77}. 
Figures~\ref{discrete_poly_20_6_05} and \ref{discrete_poly_20_8_03} below 
present such optimizing polynomials for some choices of the parameters.
The theorem is not so surprising when one recalls that we are trying to minimize the expectation
of $P$ under the positivity constraints of the class $\poly_k^d$. The theorem is valid in the
generality of Pr\'ekopa's work, i.e., the first $k$ moments of $S$ are given but they do not
necessarily equal those of a Binomial random variable.

We remark that the case in which there is more than one optimizing polynomial
is the case in which some degeneracy occurs in the problem, allowing the optimal
distribution for \eqref{one_d_max_discrete_problem} to be supported on less than $k+1$ points.
\end{section}

\begin{section}{The relaxed problem}\label{relaxed_problem_section}
As explained in the introduction, in \cite{BGGP} an upper bound for $M(n,k,p)$ was given.
The bound was proven by considering a relaxed version of problems \eqref{one_d_max_discrete_problem}
and \eqref{one_d_max_dual_discrete_problem}. In this section we describe this relaxed version
(doing so, we follow the ideas presented in \cite{BGGP}).
Problem \eqref{one_d_max_discrete_problem} is replaced by
\begin{equation}
\tilde{M}(n,k,p)=\max_{S\in\A^c(n,k,p)} \P(S=n), \label{one_d_max_continuous_problem}
\end{equation}
where $\A^c(n,k,p)$ (here the $c$ stands for continuous) is the set of all real random variables $S$ satisfying
\begin{enumerate}
\item[(I')] $S$ is supported on $[0,n]$.
\item[(II')] $\E S^i = \E X^i$ for $X\sim \Bin(n,p)$ and $1\le i\le k$.
\end{enumerate}
Comparing conditions (I), (II) above to conditions (I'), (II') here we see that the only difference between the original and relaxed problems is that in the relaxed problem $S$ may take non-integer values between $0$ and $n$. Of course, inequality \eqref{tilde_M_bounds_M} follows trivially. %we have the trivial inequality $M(n,k,p)\le \tilde{M}(n,k,p)$.
In \cite{BGGP}, the exact value of $\tilde{M}$ was found, giving the formula \eqref{value_of_tilde_M} for even $k$. 
%it was shown that for even $k$
%\begin{equation} \label{value_of_tilde_M}
%\tilde{M}(n,k,p)=\frac{p^n}{\P(\Bin(n,1-p)\le \frac{k}{2})} .
%\end{equation}
The reason that $\tilde{M}$ is easier to handle than $M$ is that the problem $\eqref{one_d_max_continuous_problem}$ is a special case of the Classical Moment Problem. Such problems have been solved, for example in the classical books \cite[Theorem 2.5.2]{Ak65}, \cite[Chap. III, sec. 3.2]{KN77}, and a great deal of theory has been developed around them.

We now consider the dual problem to \eqref{one_d_max_continuous_problem}, which is
\begin{equation} \label{one_d_max_dual_continuous_problem}
\tilde{M}(n,k,p) = \min_{P\in \poly_k^c} \E_{\Bin(n,p)} P(X) ,
\end{equation}
where $\poly_k^c$ is the collection of polynomials $P:\R\to\R$
of degree at most $k$ satisfying $P(x)\ge 0$ for $x\in[0,n)$ and $P(n)\ge 1$ (the $c$ stands for continuous).
The optimizing polynomial is explicitly given in \cite{Ak65},
it equals $1$ at $n$ and for even $k$ it has $\frac{k}{2}$ double roots in $[0,n)$ (for odd $k$ it has one root at $0$ and $\frac{k-1}{2}$ double roots in $(0,n)$). The location of the roots is given in terms of Krawtchouk polynomials,
the orthogonal polynomials of the Binomial distribution.
In Figures~\ref{continuous_poly_20_6_05} and \ref{continuous_poly_20_8_03} we
have drawn the optimizing polynomials for some specific parameters. Refer to \cite{BGGP}
for more details on the optimizing polynomials.
\begin{figure}[ht]
\begin{minipage}[t]{0.45\textwidth}
\includegraphics[width=\textwidth,clip=true]{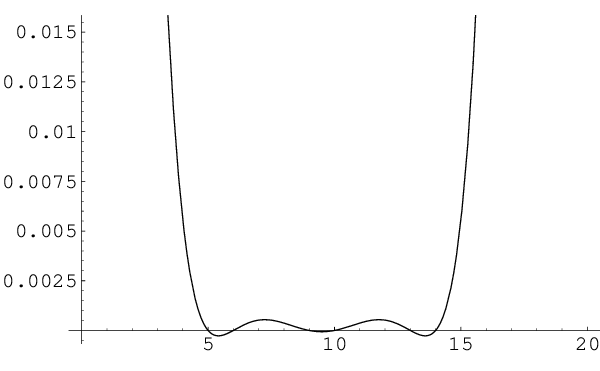}
\caption{Optimizing polynomial in \eqref{one_d_max_dual_discrete_problem} for $n=20,k=6,p=\frac{1}{2}$.}\label{discrete_poly_20_6_05}
\end{minipage}
\qquad
\begin{minipage}[t]{0.45\textwidth}
\includegraphics[width=\textwidth,clip=true]{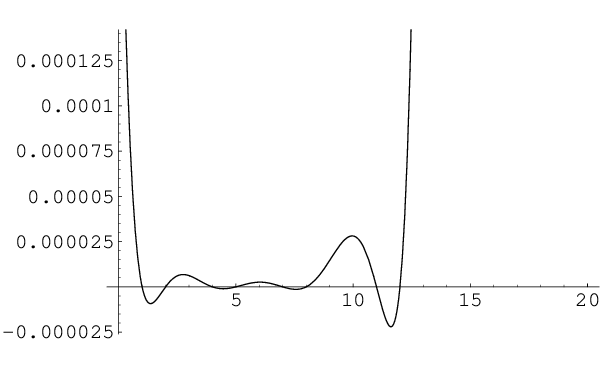}
\caption{Optimizing polynomial in \eqref{one_d_max_dual_discrete_problem} for $n=20,k=8,p=\frac{3}{10}$.}\label{discrete_poly_20_8_03}
\end{minipage}

\begin{minipage}[t]{0.45\textwidth}
\includegraphics[width=\textwidth,clip=true]{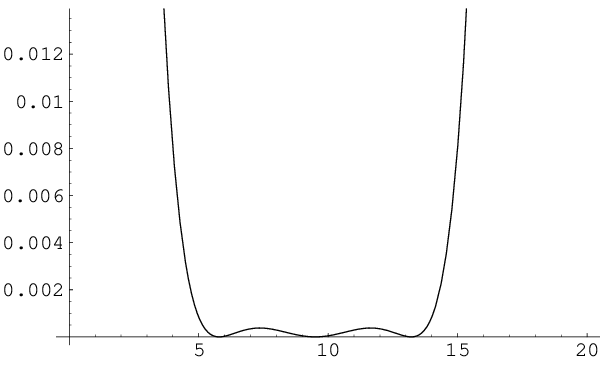}
\caption{Optimizing polynomial in \eqref{one_d_max_dual_continuous_problem} for $n=20,k=6,p=\frac{1}{2}$.}\label{continuous_poly_20_6_05}
\end{minipage}
\qquad
\begin{minipage}[t]{0.45\textwidth}
\includegraphics[width=\textwidth,clip=true]{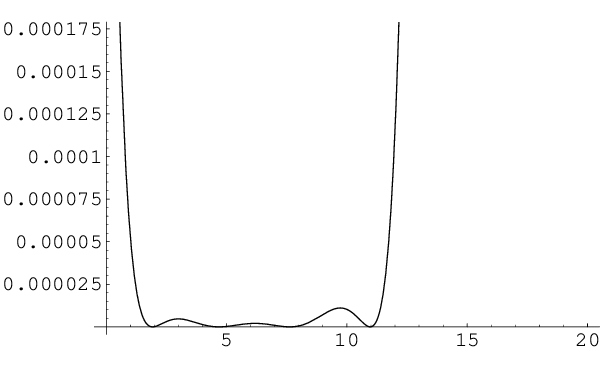}
\caption{Optimizing polynomial in \eqref{one_d_max_dual_continuous_problem} for $n=20,k=8,p=\frac{3}{10}$.}\label{continuous_poly_20_8_03}
\end{minipage}
\end{figure}

It seems worth mentioning that for even $k$ there is another problem
which is equivalent to the relaxed dual problem \eqref{one_d_max_dual_continuous_problem}.
 This other problem has been used by some authors to obtain similar upper bounds, sometimes without noting the equivalence to
 \eqref{one_d_max_dual_continuous_problem}. This equivalence is also fundamental in the analysis of the
 Classical Moment Problem. The equivalent problem for even $k$ is
\begin{equation}\label{one_d_max_dual_continuous_problem_sq_form}
\tilde{M}(n,k,p) = \min_{P\in \poly_k^2} \E_{\Bin(n,p)} P(X) ,
\end{equation}
where $\poly_k^2$ is the collection of polynomials $P:\R\to\R$ of the form $P=R^2$ where $R$ is a polynomial of degree at most $\frac{k}{2}$ satisfying $|R(n)|\ge 1$. It is clear that $\poly_k^2\subseteq \poly_k^c$ but in fact they are equal. This follows immediately from the Markov-Lukacs theorem (see for example \cite[Chap. III, thm. 2.2]{KN77})
\begin{theorem}(Markov-Lukacs) A polynomial $P$ of even degree is non-negative on $[a,b]$ iff it is of the form
\begin{equation}
P(x) = R^2(x) + (x-a)(b-x)Q^2(x)
\end{equation}
for some polynomials $Q$ and $R$.
\end{theorem}

\end{section}

\begin{section}{Proof of main result}\label{main_results_section}

In this section we show how to reduce our main result, Theorem~\ref{main_theorem}, to a result about polynomials. We also give the estimate on $\tilde{M}(n,k,p)$ required to deduce Corollary~\ref{intro_corollary} from Corollary~\ref{main_result_corollary}.
%In this section we prove our main result, Theorem~\ref{main_theorem}. 
%Recall that $M(n,k,p)$
%stands for the maximal probability of the $\AND$ event under a
%$k$-wise independent distribution on $n$ bits with marginal
%probability $p$. Recall also that $\tilde{M}(n,k,p)$ is given in
%\eqref{value_of_tilde_M} and that $M(n,k,p)\le\tilde{M}(n,k,p)$.
%Our main result is a lower bound for $M(n,k,p)$ matching the bound
%given by $\tilde{M}(n,k,p)$ up to multiplicative factors of lower
%order, in a large regime of the parameters.

Theorem \ref{main_theorem} is proved using the following general idea. 
Consider any polynomial $P\in\poly_k^d$ of the form given in Pr\'ekopa's
Theorem \ref{prekopa_theorem}. Change the location of its roots
slightly to make each pair of adjacent roots into one double root.
The new perturbed polynomial $\tilde{P}$ is in $\poly_k^c$. Show
that the expectation under the $\Bin(n,p)$ distribution of
$\tilde{P}$ is not much higher than that of $P$. Deduce that the
expectation of the optimal polynomial in
\eqref{one_d_max_dual_discrete_problem} is not much lower than the
expectation of the optimal polynomial in
\eqref{one_d_max_dual_continuous_problem}.

The actual proof that the two expectations are close is somewhat
complicated. A key ingredient is the use of discrete Chebyshev
polynomials to bound the ratio of the value of $P$ and $\tilde{P}$
at certain points.
%NEW
The Chebyshev polynomials were previously used in a similar context;
see, for example, \cite{HLL,S99}.
%END NEW

The result we need about polynomials is the following. Let $k$ be even and fix two polynomials
\begin{eqnarray}
f(x) = \prod_{i=1}^{k/2} (x-a_i) (x-a_i-1) \quad \textrm{ and }
\quad g(x) = \prod_{i=1}^{k/2} (x-a_i)^2 \label{f_and_g_def} ,
\end{eqnarray}
with all $a_i\in\{0,1,\ldots, n-1\}$ and such that $a_i\neq a_j$ and $a_i\neq a_j+1$
for $i\neq j$.

For a polynomial $\vphi$ we denote $\E_{n,p}[\vphi] = \E \br{ \vphi(X) }$ where
$X \sim \Bin(n,p)$ has Binomial distribution with parameters $n$ and $p$.

\begin{theorem} \label{thm:  ratio between polynomials}
There exist constants $c_1,c_2,c_3 > 0$ such that the following holds.
Let $n \in \N$,  $k \in \N$ even and $0 < p < 1$.
Let $N = np(1-p) - 1$.
Assume $k \leq c_1 \cdot N$.
Then,
$$ \E_{n,p} [g] \leq c_3 k \cdot
\exp \sr{ c_2 \cdot \frac{k}{V(N/k)}}
 \E_{n,p}[f] , $$
where $V(a)=\exp \Big( \sqrt{\log(a)\log\log(a)} \Big)$.
\end{theorem}
The proof of Theorem~\ref{thm:  ratio between polynomials} is given in Section~\ref{sec: pert root} below.

\begin{proof}[Proof of Theorem~\ref{main_theorem}]
By Pr\'ekopa's Theorem~\ref{prekopa_theorem},
there exists $f$ such that
$$M(n,k,p) =  \frac{\E_{n,p}[ f ]}{f(n)},$$
$f$ has the form given in \eqref{f_and_g_def}, and $\frac{f}{f(n)} \in \poly_k^d$.
By Theorem~\ref{thm:  ratio between polynomials},
$$ \E_{n,p} [g] \leq c_3 k \cdot
\exp \sr{ c_2 \cdot \frac{k}{V(N/k)}} \E_{n,p}[f] , $$
for $g$ as in \eqref{f_and_g_def}.
Since $\frac{g}{g(n)} \in \poly_k^c$,
$$ \tilde{M}(n,k,p) \leq \frac{\E_{n,p} [g]}{g(n)} , $$
which completes the proof, as $g(n) \geq f(n)$.
\end{proof}

Corollary \ref{intro_corollary} follows from Corollary~\ref{main_result_corollary} 
using the following bounds on $\tilde{M}(n,k,p)$.

\begin{claim}\label{tilde_M_estimates}
Let $n \in \N$, $k \in \N$ even, and $0 < p < 1$. There exist $C,c>0$ such that if $k\le c n(1-p)$ then
\begin{equation*}
c\sqrt{k}\sr{\frac{pk}{2e(1-p)n}}^{k/2}\le \tilde{M}(n,k,p)\le C\sqrt{k}\sr{\frac{pk}{2e(1-p)n}}^{k/2}e^{k^2/2n}.
\end{equation*}
\end{claim}

\begin{proof}%[Proof of Claim \ref{tilde_M_estimates}] 
We first recall that there exist $C_4,c_4>0$ such that $c_4\sqrt{k}\sr{\frac{k}{2e}}^{k/2}\le \sr{\frac{k}{2}}!\le C_4\sqrt{k}\sr{\frac{k}{2e}}^{k/2}$. Hence
\begin{equation*}
\Choose{n}{k/2}\le \frac{n^{k/2}}{(k/2)!}\le \frac{1}{c_4\sqrt{k}}\sr{\frac{2en}{k}}^{k/2}
\end{equation*}
and since $k\le c n(1-p)$ (for a small enough $c$), we have
\begin{equation*}
\Choose{n}{k/2}\ge \frac{(n-k/2)^{k/2}}{(k/2)!}\ge \frac{1}{C_4\sqrt{k}}\sr{\frac{2en}{k}}^{k/2}e^{-k^2/2n}.
\end{equation*}
Hence
\begin{equation*}
\begin{split}
\P\left(\Bin(n,1-p)\le \frac{k}{2}\right)&\ge \Choose{n}{k/2}p^{n-k/2}(1-p)^{k/2}\ge\\
&\ge \frac{p^n}{C_4\sqrt{k}}\sr{\frac{2e(1-p)n}{pk}}^{k/2}e^{-k^2/2n}.
\end{split}
\end{equation*}
Similarly note that since $k\le c n(1-p)$, we have
\begin{equation*}
\Choose{n}{i}p^{n-i}(1-p)^i\le \frac{1}{2}\Choose{n}{i+1}p^{n-i-1}(1-p)^{i+1}
\end{equation*}
for $i\le k/2$. Hence
\begin{equation*}
\P\left(\Bin(n,1-p)\le\frac{k}{2}\right)\le 2\Choose{n}{k/2}p^{n-k/2}(1-p)^{k/2}\le \frac{2p^n}{c_4\sqrt{k}}\sr{\frac{2e(1-p)n}{pk}}^{k/2}.
\end{equation*}

The claim now follows by substituting the above estimates into \eqref{value_of_tilde_M}.
\end{proof}

\section{Perturbing Roots of Polynomials}
\label{sec: pert root}
In this section we shall prove Theorem~\ref{thm:  ratio between polynomials}.
For $n \in \N$, we denote $\intn = \set{0,1,\ldots,n}$.
For two real numbers $a$ and $b$, we denote $[a,b) = \sett{t}{a \leq t < b}$ and
$[a,b] = \sett{t }{a \leq t \leq b}$.
For given $n\in\N$ and $0<p<1$, we define $\Pr_{n,p} [x] = \Choose{n}{x}p^x (1-p)^{n-x}$ for $x\in \intn$;
i.e., the probability of $x$ according to $\Bin(n,p)$.

We wish to bound the ratio between $\E_{n,p} \br{g}$ and $\E_{n,p} \br{f}$. We write
\begin{equation}\label{basic_ratio}
\frac{\E_{n,p} \br{g}}{\E_{n,p} \br{f}} = \frac{\sum_{x=0}^n \Pr_{n,p}[x]g(x)}{\sum_{x=0}^n \Pr_{n,p}[x]f(x)}.
\end{equation}
The theorem then follows from the following two lemmas:
\begin{lem}
\label{lem: non zero x}
Let $x \in \intn$ be such that $f(x) \neq 0$.
Then $$\frac{g(x)}{f(x)} \leq 2 \sqrt{k}.$$
\end{lem}
\begin{lem} \label{lem: there exists far w}
There exist universal constants $c_1,c_2,c_3 > 0$ such that the
following holds. Let $N = np(1-p) - 1$.
Assume $k \leq c_1 \cdot N$. Then, for every
$x\in\intn$ there exists $w \in \intn$ satisfying
\begin{equation}\label{one_point_expect_ratio_bound}
\frac{\Pr_{n,p}[x]g(x)}{\Pr_{n,p}[w]f(w)}\le
c_3 \cdot \exp \sr{  c_2 \cdot \frac{k}{V(N/k)}},
\end{equation}
where $V(a)=\exp \Big( \sqrt{\log(a)\log\log(a)} \Big)$.
\end{lem}

The first lemma, whose proof is much simpler than the 
proof of the second lemma,
is proved in Section~\ref{proof_of_the_easy_ratio_bound_lemma}.
The second lemma addresses the case $f(x)=0$ in which
the first lemma does not apply,
and is proved in Section~\ref{proof_of_harder_ratio_bound_lemma}.
We note that the `simple' ideas presented in the proof of the first lemma
can yield a weaker version of the second lemma,
with a bound of the form $c_3\exp(c_2 k)$ on the RHS of \eqref{one_point_expect_ratio_bound}. While significantly weaker, such a bound still yields the correct asymptotic behavior of $M(n,k,p)$ for constant $k$.
%with worst bounds for a smaller range of parameters.

We now show how Theorem \ref{thm:  ratio between polynomials} follows from the two lemmas.
\begin{proof}[Proof of Theorem~\ref{thm:  ratio between polynomials}]
Let $$\Ze(f) = \sett{y \in \intn}{f(y) = 0}.$$
We shall denote the $w$ that corresponds to $x$ according to Lemma \ref{lem: there exists far w} by $w_x$.
We write \eqref{basic_ratio} using the above two lemmas and using the fact that for all $x\in\intn$,
$\Pr_{n,p}[x]\ge 0$ and $f(x),g(x)\ge0$ (due to the special structure \eqref{f_and_g_def}
of the polynomials) as
\begin{eqnarray*}
\E_{n,p}[g]
& \le & 2\sqrt{k} \sum_{x\in \intn\setminus\Ze(f) } \Pr_{n,p}[x]f(x) \\
&& \qquad + c_3 \cdot \exp \sr{  c_2 \cdot \frac{k}{V(N/k)}} \sum_{x\in\Ze(f)} \Pr_{n,p}[w_x]f(w_x) \\
& \le & (c_3+2) k \cdot
\exp  \sr{  c_2 \cdot \frac{k}{V(N/k)}} \E_{n,p}[f].
\end{eqnarray*}
\end{proof}

\subsection{Points that are not zeros of $f$}\label{proof_of_the_easy_ratio_bound_lemma}
\begin{proof}[Proof of Lemma \ref{lem: non zero x}]
Note that
\begin{equation}
\frac{g(x)}{f(x)} = \prod_{i=1}^{k/2} \frac{a_i-x}{a_i+1-x}.\label{ratio_expression_at_a_point}
\end{equation}
Since $f(x) \neq 0$, we can partition the $a_i$'s into two sets:
$$ S_1 = \sett{i}{a_i < x} \ \ \text{and} \ \ S_2 = \sett{i}{a_i > x}.$$
First, for each $i\in S_2$
\begin{eqnarray} \label{eqn: first}
0 \leq \frac{a_i-x}{a_i+1-x} \leq 1.
\end{eqnarray}
In addition,
\begin{align}
 0  \leq  \prod_{i \in S_1} \frac{x-a_i}{x-a_i-1}
 & \leq \prod_{i=1}^{k/2} \Big( 1+ \frac{1}{2i-1} \Big)
 \leq  2 \ \textrm{exp} \Big( \sum_{i=2}^{k/2} \frac{1}{2i-1} \Big) \nonumber \\
  &
  \leq  2 \ \textrm{exp} \Big( \frac{1}{2} \log (k-1) \Big)
 \leq  2 \sqrt{k}. \label{eqn: second}
\end{align}
The lemma follows by substituting \eqref{eqn: first} and \eqref{eqn: second} in \eqref{ratio_expression_at_a_point}.
\end{proof}

\subsection{Points that are zeros of $f$}\label{proof_of_harder_ratio_bound_lemma}

In this section we prove Lemma~\ref{lem: there exists far w}.  We first describe a family of
orthogonal polynomials, the discrete Chebyshev polynomials.
Then we prove Claim~\ref{clm: w far from x} that uses these polynomials.
Finally we use the Claim~\ref{clm: w far from x} to prove Lemma~\ref{lem: there exists far w}.

\subsubsection{Orthogonal Polynomials}\label{orthogonal_polynomials_subsubsection}

We now give some properties of a family of orthogonal
polynomials studied by Chebyshev, sometimes called discrete
Chebyshev polynomials. These properties are described and proved in \cite[Section
2.8]{Sz75}. We use these orthogonal polynomial to prove the following proposition.

\begin{prop}\label{lower_bound_for_G_cor}
Let $M \in \N$ and let $G$ be a monic polynomial of degree $d$ for $0\le d\le \frac{M}{2}$, then
\begin{equation*}
\max_{i\in \set{0,\ldots,M-1}} |G(i)|\ge
\frac{M^d}{4^{d+1/2}}e^{-d^3/M^2}.
\end{equation*}
\end{prop}

\begin{proof}
The family of polynomials $\{t_d\}_{d=0}^{M-1}$
defined below are orthogonal polynomials for the measure $\mu_M$
which assigns mass one to each integer $x\in\{0,1,\ldots,M-1\}$,
see \eqref{eqn: sum of tn} for the chosen normalization. In other
words for every $d,d'\in \set{0,\ldots,M-1}$ such that $d\neq d'$,
$$ %\int t_d(x)t_{d'}(x)d\mu_M(x)=
\sum_{i = 0}^{M-1} t_d(i)t_{d'}(i) = 0.$$
The polynomial $t_d$ is
\begin{eqnarray}
\label{eqn: dfn of tn}
t_d(x) = d! \cdot \Delta^{(d)} \Choose{x}{d} \Choose{x-M}{d},
\end{eqnarray}
where 
$$\Delta G(x) \eqdef G(x+1) - G(x) \ , \quad \Delta^{(d)} G \eqdef \Delta \br{ \Delta^{(d-1)} G } $$
and
$$ \Choose{x}{d} \eqdef \frac{x (x-1) \cdots (x-d+1)}{d!}.$$
The normalization is chosen so that
\begin{eqnarray}
\label{eqn: sum of tn}
\sum_{i=0}^{M-1} \abs{t_d(i)}^2 = \frac{M(M^2-1^2)(M^2-2^2)\cdots (M^2-d^2)}{2d+1}.
\end{eqnarray}
The coefficient of $x^{2d}$ in $d! \cdot \Choose{x}{d} \Choose{x-M}{d}$ is $\frac{1}{d!}$.
Thus, by the linearity of $\Delta$, and since for every $k \in \N$,
$$\Delta x^k = (x+1)^k - x^k = k x^{k-1} + \Choose{k}{2} x^{k-2} + \cdots + 1,$$
the polynomial $t_d$ has degree $d$, and the coefficient of $x^d$
in $t_d$ is $\Choose{2d}{d}$. Thus, since every monic polynomial
$G$ of degree $d<M$ can be expanded as $G(x) =
\Choose{2d}{d}^{-1}t_d(x)+\sum_{i=0}^{d-1} a_i t_i(x)$, we have
using \eqref{eqn: sum of tn}
\begin{equation}
\begin{split}
\label{eqn: tn minimize}
\sum_{i=0}^{M-1} \abs{G(i)}^2 &\geq
\Choose{2d}{d}^{-2}\sum_{i=0}^{M-1} \abs{t_d(i)}^2=\\
&=\Choose{2d}{d}^{-2} \frac{M(M^2-1^2)(M^2-2^2)\cdots (M^2-d^2)}{2d+1}.
\end{split}
\end{equation}
Using the inequalities $1-x\ge e^{-2x}$ ($0\le x\le \frac{1}{4}$),
$\Choose{2d}{d}\le 2 \frac{4^d}{\sqrt{\pi d}}$ ($d\ge 1$) and
$\sum_{i=1}^d i^2=\frac{d(d+1)(2d+1)}{6}\le d^3$ ($d\ge 1$) we
obtain for $1\le d \leq \frac{M}{2}$,
\begin{equation}
\sum_{i=0}^{M-1} \abs{G(i)}^2\ge \frac{\pi
dM^{2d+1}}{4^{2d+1}(2d+1)}e^{-2\sum_{i=1}^d i^2/M^2}\ge
\sr{\frac{M}{4}}^{2d+1}e^{-2d^3/M^2}.
\end{equation}
The proposition thus follows (the case $d=0$ is straightforward).
\end{proof}

\subsubsection{A Segment With Few Zeros}\label{segment_with_few_zeros_subsubsection}

In this section we prove an auxiliary claim, to be used in the
next section as a main component in the proof of Lemma~\ref{lem:
there exists far w}. The claim roughly states that given a segment
with few zeros, we can find a point at which $f$ obtains a `large'
value.

\begin{claim}
\label{clm: w far from x}
Let $x\in\intn$, let $R,m
\in \N$ be such that $m \geq 2 R$,
let $L$ be a non-negative integer, and let $\tau > 4$.
If
\begin{equation}\label{few_zeros_on_right}
\abs{ \Ze(f) \cap [x+m,x+ \tau m) } \leq R ,
\end{equation}
and
\begin{equation} \label{many zeros on left}
\abs{ \Ze(f) \cap (x,x+m/\tau) } \geq L ,
\end{equation}
then there exists $w \in \N \cap [x+2m, x+3m]$ such that
\begin{eqnarray*}
\frac{g(x)}{f(w)} \leq 8\cdot
\exp \Big( \frac{12k}{\tau} + 6R - L \log \tau \Big).
\end{eqnarray*}
Similarly, if instead of \eqref{few_zeros_on_right} and \eqref{many zeros on left} we have
\begin{equation*}
\abs{ \Ze(f) \cap (x- \tau m,x-m] } \leq R ,
\end{equation*}
and
\begin{equation*}
\abs{ \Ze(f) \cap (x-m/\tau,x) } \geq L ,
\end{equation*}
then there exists $w \in \N \cap [x-3m, x-2m]$ such that
\begin{eqnarray*}
\frac{g(x)}{f(w)} \leq 8\cdot
\exp \Big( \frac{12k}{\tau} + 6R - L \log \tau \Big).
\end{eqnarray*}
\end{claim}

\begin{proof}
Assume without loss of generality that
\eqref{few_zeros_on_right} and \eqref{many zeros on left} hold
(a similar argument holds for the second case). 
Let $x+2m\le w \le x+3m$ be such that
$f(w) \neq 0$.
Write
\begin{equation}\label{basic_ratio_second_lemma}
\frac{g(x)}{f(w)} = \prod_{i=1}^{k/2} \frac{(x-a_i)^2}{(w-a_i)(w-a_i-1)}.
\end{equation}
We partition the $a_i$'s into six subsets $S_1, S_2,\ldots,S_6$ 
according to the definitions below, and bound
\eqref{basic_ratio_second_lemma} over each subset separately.
The partition is
\begin{align*}
S_1&=\sett{ i \in \{1,\ldots,k/2\}}{ a_i < x} ,\\
S_2&=\sett{ i \in \{1,\ldots,k/2\}}{ x \leq a_i < x + m/\tau } ,\\
S_3&=\sett{ i \in \{1,\ldots,k/2\}}{ x+ m/\tau \leq a_i < x + m } ,\\
S_4&= \sett{i \in \{1,\ldots,k/2\}}{x+m\le a_i< x+4m} ,\\
S_5&= \sett{i \in \{1,\ldots,k/2\}}{x+4m \le a_i < x+\tau m}, \ \text{and} \\
S_6&= \sett{i \in \{1,\ldots,k/2\}}{x+ \tau m \leq a_i} .
\end{align*}

For every $i\in S_1 \cup S_3$,
\begin{eqnarray*}
0\leq \Big| \frac{x-a_i}{w-a_i-1} \Big| \leq 1,
\end{eqnarray*}
which implies
\begin{eqnarray}\label{first_second set_bound}
0\leq \prod_{i\in S_1 \cup S_3}\frac{(a_i-x)^2}{(w-a_i)(w-a_i-1)} \leq 1.
\end{eqnarray}

For every $i\in S_2$,
\begin{eqnarray*}
0\leq \frac{a_i-x}{w-a_i-1} \leq \frac{1}{\tau},
\end{eqnarray*}
which implies
\begin{eqnarray}\label{first_set_bound}
0\leq \prod_{i\in S_2}\frac{(a_i-x)^2}{(w-a_i)(w-a_i-1)} \leq \tau^{-L} .
\end{eqnarray}

To argue about $S_4$, define the polynomial
$$ F(\xi) \eqdef \prod_{i \in S_4} (\xi-a_i)(\xi-a_i-1). $$
Denote $d=2|S_4|$, the degree of $F$. 
Since $2 |S_4| \le R \le \frac{m}{2}$,
we deduce from Proposition~\ref{lower_bound_for_G_cor} that there
exists $w_0 \in \N \cap [x+2m,x+3m]$ such that
$$ \abs{F(w_0)} \ge \frac{m^{d}}{4^{d+1/2}} e^{-d^3/m^2}.$$
Hence, since $\prod_{i\in S_4} (a_i-x)^2 \leq (4m)^d$,
\begin{equation} \label{fourth_set_bound}
0\leq \prod_{i\in S_4}\frac{(a_i-x)^2}{(w_0-a_i)(w_0-a_i-1)}
\le 2\cdot e^{3R+R^3/m^2} \le 2\cdot e^{4R}.
\end{equation}

For every $i\in S_5$,
\begin{equation*}
0\leq\frac{a_i-x}{a_i-w}\le 4.
\end{equation*}
Hence, since $|S_5| \leq \frac{R+1}{2}$,
\begin{equation}\label{second_set_bound}
0\leq \prod_{i\in S_5}\frac{(a_i-x)^2}{(a_i-w)(a_i+1-w)} \le 4 \cdot 4^{R}.
\end{equation}

Similarly, since $2 \abs{S_6} \leq k$,
\begin{eqnarray}\label{third_set_bound}
0 \leq \prod_{i \in S_6} \frac{(a_i-x)^2}{(a_i-w)(a_i+1-w)} \leq
\left(\frac{\tau}{\tau-3} \right)^{2|S_6|}\le \exp \Big( \frac{12k}{\tau} \Big).
\end{eqnarray}

Therefore, plugging \eqref{first_set_bound}, \eqref{first_second set_bound}, \eqref{second_set_bound},
\eqref{third_set_bound} and \eqref{fourth_set_bound} into \eqref{basic_ratio_second_lemma}
with $w=w_0$,
\begin{equation*}
\frac{g(x)}{f(w_0)} \leq 8\cdot
\exp \Big( \frac{12k}{\tau} + 6R - L \log \tau \Big).
\end{equation*}
\end{proof}

\subsubsection{Finding good $w$}\label{finding_good_w_subsubsection}

The following claim shows that there exists a $w$ that is `close'
to $x$ on which $f$ obtains a `large' value.

\begin{claim}
\label{clm: good w} Let $x\in\intn$, $k\ge Z \in \N$,
and $\tau\ge e^{12}$.
Let $K$ be the smallest integer such that
$$\Big\lfloor \frac{\log \tau}{6} \Big\rfloor^{\frac{K-1}{2}} \geq \frac{k}{Z}.$$
Then, there exist integers $w_1 > x$ and $w_2 < x$ such that for
each $w \in \set{w_1,w_2}$,
$$3 Z \leq |w  - x| \leq 9 Z \tau^{K}$$
and
\begin{eqnarray*}
\frac{g(x)}{f(w)} \leq 8 \cdot \exp \Big( \frac{12k}{\tau} + 6Z \Big).
\end{eqnarray*}
\end{claim}

\begin{proof}
We show the existence of $w_1$, the existence of $w_2$ can be shown
similarly.

Let $Z_0 = Z_1 = Z$.
Let $m_0 = 3 Z_0$ and $m_1 = \tau m_0$.
If either 
\begin{align}
\label{eqn: 1 in clm}
\abs{\Ze(f) \cap [x+m_0,x+m_1)} \leq Z_0
\end{align}
or
\begin{align}
\label{eqn: 2 in clm} 
\abs{\Ze(f) \cap [x+m_1,x+m_2)} \leq Z_1 ,
\end{align}
then by Claim \ref{clm: w far from x}, with
$L=0$, $R=Z$, $m = m_0 \geq 2Z$ for \eqref{eqn: 1 in clm}
and $m = m_1 \geq 2 Z$ for \eqref{eqn: 2 in clm}, there exists $w_1 \in \N$ such that
$$ 3 Z \leq w_1  - x \leq 9 Z \tau$$
and
\begin{eqnarray*}
\frac{g(x)}{f(w_1)} \leq 8\cdot \exp \Big( \frac{12k}{\tau} + 6 Z \Big) .
\end{eqnarray*}

Thus, assume that both \eqref{eqn: 1 in clm}
and \eqref{eqn: 2 in clm} do not hold.
Define $Z_i$ and $m_i$ for $i \geq 2$ as
$$Z_i = \Big\lfloor \frac{\log \tau}{6} \Big\rfloor Z_{i-2}  \ \ \text{and} \ \ m_i = \tau m_{i-1}.$$
Since the intervals
$$[x+m_0,x+m_1),\ldots,[x+m_{K},x+m_{K+1})$$ are disjoint,
since the number of zeros of $f$ is $k$, and since $Z_K \geq k$,
let $i$ be the smallest integer so that
$$\abs{\Ze(f) \cap [x+m_i,x+m_{i+1})} \leq Z_i $$
and
$$\abs{\Ze(f) \cap (x,x+m_{i-1})} \geq Z_{i-2}.$$
Since $m_i \ge 2 Z_i$, by Claim~\ref{clm: w far from x}
with $L=Z_{i-2}$, $R=Z_i$ and $m =m_i$,
there exists $w_1 \in \N$ such that
$$ 3 Z \leq w_1  - x \leq 9 Z \tau^{K}$$
and
\begin{eqnarray*}
\frac{g(x)}{f(w_1)} \leq 8\cdot
\exp \Big( \frac{12k}{\tau} + 6Z_{i} - Z_{i-2} \log \tau \Big)
\leq 8 \cdot \exp \Big( \frac{12k}{\tau}\Big).
\end{eqnarray*}
\end{proof}

\subsubsection{Probability Estimates}\label{probability_estimates_subsubsection}

\begin{clm} \label{clm:  prob. estimates}
Let $\ell \in \{1,\ldots,n\}$, $0 < p < 1$,
and let $N = np(1-p) - 1$.
Assume $N > 0$.
Set $\mu = \lfloor p n \rfloor$.
If $\ell \leq
(n-\mu)/2$, then
$$ \Pr_{n,p} \br{ \mu + \ell } \geq \exp \sr{ - \frac{3\ell^2}{2N}  } \Pr_{n,p} [\mu]   . $$
In addition, if $\ell \leq \mu/2$, then
$$ \Pr_{n,p} \br{ \mu - \ell }  \geq \exp \sr{ - \frac{8 \ell^2}{N}  }  \Pr_{n,p} [\mu] . $$
\end{clm}

\begin{proof}
Assume that $\ell-1 \leq (n-\mu)/2$. Then,
\begin{eqnarray*}
    \frac{\Pr_{n,p} [\mu]}{\Pr_{n,p} [\mu +\ell] } % & = & \sr{ \frac{1-p}{p} }^\ell \cdot
    & = & \sr{ \frac{1-p}{p} }^\ell \cdot
    \frac{ \mu^\ell \prod_{i =1}^\ell (1+i/\mu)}{(n-\mu)^\ell\prod_{i=0}^{\ell-1}(1-i/(n-\mu))} \\
    & \leq & \sr{ \frac{1-p}{p} }^\ell \cdot
    \frac{ \mu^\ell }{(n-\mu)^\ell} \exp \sr{ \sum_{i =1}^\ell \frac{i}{\mu}  + \frac{2(i-1)}{n-\mu} } \\
    & = & \sr{ \frac{1-p}{p} }^\ell \cdot
    \frac{ \mu^\ell }{(n-\mu)^\ell} \exp \sr{ \ell \cdot \sr{ \frac{(\ell+1)(n-\mu)+2\mu(\ell-1)}{2\mu(n-\mu)} } } \\
    & \leq & \exp \sr{ \ell \cdot \sr{ \frac{\ell(n+\mu)+n}{2\mu(n-\mu)} } } \\
    & \leq & \exp \sr{ \frac{3\ell^2}{2n} \cdot \sr{ \frac{1}{p(1-p) - 1/n} } }.
\end{eqnarray*}
This proves the first assertion.
For the second assertion note that
$$ \ell \leq \frac{\mu}{2} \leq \frac{\lceil pn \rceil}{2} = \frac{n - \lfloor (1-p)n \rfloor}{2} . $$
Recall that the binomial measure decreases as the distance from its expectation increases. Thus,
\begin{equation*}
\Pr_{n,p} [ \mu- \ell ] = \Pr_{n,(1-p)} [n-\mu+\ell]
\geq \Pr_{n,(1-p)} [\lfloor (1-p)n \rfloor + 1 +\ell ] .
\end{equation*}
In addition, since $(\ell+1) - 1 \leq \frac{n - \lfloor (1-p)n \rfloor}{2}$, the proof of the first assertion implies
\begin{eqnarray*}
\Pr_{n,p} [\mu] & \leq & \exp \sr{ \frac{3}{2N} } \cdot \Pr_{n,p} [\mu+1] \\
& \leq & \exp \sr{\frac{3}{2N} } \cdot \Pr_{n,p} [n - \lfloor (1-p)n \rfloor ] \\
& = & \exp \sr{ \frac{3}{2N} } \cdot \Pr_{n,(1-p)} [\lfloor (1-p)n \rfloor ] \\
& \leq & \exp \sr{ \frac{3}{2N} } \cdot \exp \sr{ \frac{3(\ell+1)^2}{2N} } \cdot
\Pr_{n,(1-p)} \br{ \lfloor (1-p)n \rfloor + 1 + \ell } \\
& \leq & \exp \sr{  \frac{3((\ell+1)^2+1)}{2N} } \cdot
\Pr_{n,p} [\mu-\ell] ,
\end{eqnarray*}
which completes the proof since $\ell \geq 1$.
\end{proof}

\subsubsection{Proof of Lemma~\ref{lem: there exists far w}}\label{proof_of_far_w_lemma_subsubsection}

\begin{proof}
Let $x\in I_n$ and $\mu=\lfloor pn\rfloor$. Let $k\ge Z \in
\N$ and $\tau \geq e^{12}$, to be determined. Let $K = K(\tau,Z)$ be the smallest integer such that
$$\Big\lfloor \frac{\log \tau}{6} \Big\rfloor^{\frac{K-1}{2}} \geq \frac{k}{Z}.$$
First assume $x<\mu$. We use Claim~\ref{clm: good w} to find $w >
x$ such that $w \leq x + 9 Z \tau^{K}$ and
\begin{eqnarray*}
\frac{g(x)}{f(w)} \leq 8\cdot \exp \Big( \frac{12k}{\tau} + 6 Z \Big) .
\end{eqnarray*}
Now if $w<\mu$ we certainly have $\Pr_{n,p}[x]\le\Pr_{n,p}[w]$.
If $w>\mu$ then we can use Claim~\ref{clm:  prob. estimates}
with $\ell=w-\mu\leq w-x \leq 9 Z \tau^K$, provided that $\ell \le \frac{n-\mu}{2}$, to obtain
\begin{equation*}
\frac{\Pr_{n,p}[x]}{\Pr_{n,p}[w]}\le \exp \sr{ \frac{8 (9Z\tau^K)^2}{N}}.
\end{equation*}
Thus, %for any value of $w$,
\begin{equation*}% \label{final_ratio_bound}
\frac{\Pr_{n,p}[x]g(x)}{\Pr_{n,p}[w]f(w)}\le 8
\cdot \exp \sr{ \frac{8 (9Z\tau^K)^2}{N} + \frac{12 k}{\tau} + 6Z }.
\end{equation*}
This also holds for $x\ge \mu$, by using Claim~\ref{clm: good w} to find $w<x$,
and the estimate in Claim~\ref{clm:  prob. estimates} involving $\Pr_{n,p}[\mu-\ell]$,
provided that $\ell \leq \frac{\mu}{2}$.

Set
$$\tau = \frac{1}{100}\exp ( \sqrt{\log(N/k)\log\log(N/k)} ) \ \ \text{and} \  \ Z = \Big\lceil \frac{k}{\tau} \Big\rceil.$$
Since $c_1$ is small enough (recall that $k \leq c_1 \cdot N$), $\tau \geq e^{12}$ and
a short calculation shows that
$$K \leq \frac{1}{4}\sqrt{\frac{\log (N/k)}{\log \log (N/k)}}.$$
Since $\ell\le 9Z\tau^K$, this implies that $\ell \leq \frac{N}{2} \leq \min \set{ \frac{\mu}{2} , \frac{n-\mu}{2}}$
(for small enough $c_1$).
Thus,
\begin{eqnarray*}
\frac{\Pr_{n,p}[x]g(x)}{\Pr_{n,p}[w]f(w)} & \le &
8\cdot \exp \sr{ c_4 \sr{\frac{k}{\tau} + 1} } ,
\end{eqnarray*}
for a constant $c_4>0$, since $\frac{Z\tau^{2K}}{N}\le \sqrt{\frac{k}{N}}\le 1$ and $Z\le \frac{k}{\tau}+1$. The lemma follows.
\end{proof}
\end{section}

\begin{section}{Open Problems}\label{open_problems_section}
\begin{enumerate}
\item What is the value of $M(n,k,p)$ in the range of the parameters not treated by our theorem, namely $k\ge Cnp(1-p)$?
\item What is the actual ratio of $M(n,k,p)$ and $\tilde{M}(n,k,p)$?
\item Is there also a similarity between the optimal distributions of our original and relaxed problems (problems \eqref{one_d_max_discrete_problem} and \eqref{one_d_max_continuous_problem})? As explained in Section~\ref{problem_and_dual_section}, this is related to whether the optimizing polynomials in the dual problems are similar. As hinted by Figures~\ref{discrete_poly_20_6_05}-\ref{continuous_poly_20_8_03}, calculations in particular cases seem to indicate this to be the case. The similarity seems especially strong in the case $p=\frac{1}{2}$.
\item In the setting of Theorem~\ref{thm:  ratio between polynomials}, What is the best ratio between $\E_{n,p}[g]$ and $\E_{n,p}[f]$? I.e., the best bound on the change in the expectation of the polynomial after small perturbation of its zeros.
\item Find upper and lower bounds for the maximal probability that all the bits are 1, for the class of \emph{almost} $k$-wise independent distributions. Similarly to $k$-wise independent distributions, such distributions have also proven quite useful for the derandomization of algorithms in computer science.
\end{enumerate}
\end{section}

{\bf Acknowledgement.} We would like to thank Itai Benjamini,
Ori Gurel-Gurevich and Simon Litsyn for several useful discussions on this problem. Part of
this work was conducted while the authors participated in the PCMI
Graduate Summer School at Park City, Utah, July 2007.

\begin{section}{Appendix}
We provide here short proofs for the results of \cite{BGGP} that we use.
\begin{proof}[Proof of \eqref{odd_case_equality}]
Fix $n\in\N, 0<p<1$ and an odd $k\in\N$. Let $P$ be the optimal polynomial for the problem \eqref{one_d_max_dual_discrete_problem} for these $n,k$ and $p$. By the second part of Theorem \eqref{prekopa_theorem} we know that
\begin{equation}\label{P_poly_form}
P(z) = \frac{z\prod_{i=1}^{(k-1)/2}(z-z_{2i})(z-z_{2i+1})}{n\prod_{i=1}^{(k-1)/2}(n-z_{2i})(n-z_{2i+1})}.
\end{equation}
Now note that
\begin{equation}\label{P_Q_equality}
\begin{split}
\E&_{\Bin(n,p)} P(Z) \\
&= \sum_{z=0}^n \frac{z\prod_{i=1}^{(k-1)/2}(z-z_{2i})(z-z_{2i+1})}{n\prod_{i=1}^{(k-1)/2}(n-z_{2i})(n-z_{2i+1})}\Choose{n}{z}p^z (1-p)^{n-z}\\
&=\sum_{z=1}^n \frac{\prod_{i=1}^{(k-1)/2}(z-z_{2i})(z-z_{2i+1})}{\prod_{i=1}^{(k-1)/2}(n-z_{2i})(n-z_{2i+1})}\Choose{n-1}{z-1}p^z (1-p)^{n-z}\\
&=p\sum_{z=0}^{n-1} \frac{\prod_{i=1}^{(k-1)/2}(z-(z_{2i}-1))(z-(z_{2i+1}-1))}{\prod_{i=1}^{(k-1)/2}(n-z_{2i})(n-z_{2i+1})}\Choose{n-1}{z}p^z (1-p)^{n-1-z}\\
&=p \cdot \E_{\Bin(n-1,p)} Q(Z),
\end{split}
\end{equation}
where
\begin{equation}\label{Q_poly_form}
Q(z) =
\frac{\prod_{i=1}^{(k-1)/2}(z-(z_{2i}-1))(z-(z_{2i+1}-1))}{\prod_{i=1}^{(k-1)/2}(n-1-(z_{2i}-1))(n-1-(z_{2i+1}-1))}.
\end{equation}
Note that $Q$ is of degree $k-1$ and satisfies $Q(i)\ge 0$ for
$i\in\{0,1,\ldots,n-2\}$, and $Q(n-1)=1$. Hence,
\begin{equation*}
M(n,k,p)\ge pM(n-1,k-1,p).
\end{equation*}
To prove that equality holds, we can carry the above reasoning
in the reverse direction by starting with the optimal polynomial $Q$ to problem \eqref{one_d_max_dual_discrete_problem},
which, by Theorem \eqref{prekopa_theorem}, is of the form \eqref{Q_poly_form}.
Then noting that \eqref{P_Q_equality}
still holds for a polynomial $P$ of the form \eqref{P_poly_form},
which is of degree $k$ and satisfies $P(i)\ge 0$ for $i\in\{0,1,\ldots,n-1\}$, and $P(n)=1$.
\end{proof}
\begin{proof}[Proof of Lemma \ref{lem: S and As are same}]
Define a distribution $\Q_S$ on $\{0,1\}^n$ by
\begin{equation*}
Q_S(\{x\}) = \P(S=|x|) \cdot {n \choose {|x|}}^{-1}
\end{equation*}
for $x\in\{0,1\}^n$, where $|x|$ is the number of $1$'s in $x$.
By definition, $Q_S$ is symmetric and $S$ has the distribution of the number of $1$'s in $\Q_S$.
It remains to verify that each bit has marginal probability $p$, and the $k$-wise independence property.
Let $\tilde{S}$ be a random variable with the $\Bin(n,p)$ distribution;
i.e., $\P(\tilde{S} = i) = \Choose{n}{i}p^i(1-p)^{n-i}$.
It is straight-forward to verify that $\Q_{\tilde{S}}$ is the distribution of
$n$ independent $\text{Bernoulli}(p)$ random variables.
Fix $1\le i_1<i_2<\cdots <i_k\le n$ and $y_1,\ldots,y_k\in\{0,1\}$.
Let $j$ be the number of 1's in $(y_1,\ldots, y_k)$. Note that
\begin{equation*}
\begin{split}
\Q_S\big( \{ x &  \in \set{0,1}^n \ :  \
x_{i_1}= y_1,\ldots, x_{i_k}=y_k \} \big) = \sum_{i=j}^{n-k+j} \Choose{n-k}{i-j}\frac{\P(S=i)}{\Choose{n}{i}}\\
&=\frac{(n-k)!}{n!}\sum_{i=j}^{n-k+j}\P(S=i)\prod_{m=0}^{j-1}(i-m)\prod_{m=0}^{k-j-1}(n-m-i) = \E P_j(S),
\end{split}
\end{equation*}
where $P_j$ is defined by
\begin{equation*}
P_j(z) = \frac{(n-k)!}{n!}\prod_{m=0}^{j-1}(z-m)\prod_{m=0}^{k-j-1}(n-m-z).
\end{equation*}
Since $P_j$ is a polynomial of degree $k$ and since $S$ has the same first $k$ moments as $\tilde{S}$,
\begin{equation*}
\begin{split}
\Q_S(\{ x &  \in \set{0,1}^n \ :  \ x_{i_1}= y_1,\ldots, x_{i_k}=y_k \})= \E P_j(S) \\ & = \E P_j(\tilde{S})
= \Q_{\tilde{S}}(\{ x  \in \set{0,1}^n \ :  \ x_{i_1}= y_1,\ldots, x_{i_k}=y_k \}),
\end{split}
\end{equation*}
as required.
\end{proof}
\begin{proof}[Proof of \eqref{value_of_tilde_M}]
Following the methods of the classical moment problem,
we use the equivalence of \eqref{one_d_max_dual_continuous_problem} and
\eqref{one_d_max_dual_continuous_problem_sq_form} and solve the
latter problem. Fix $n\in\N, 0<p<1$ and an even $k\in\N$. Let
$P=R^2$ for a polynomial $R$ of degree at most $k/2$ satisfying
$|R(n)|\ge 1$. Let $\{K_i\}_{i=0}^n$ be the $(n,p)$-Krawtchouk
polynomials; i.e., the orthogonal polynomials corresponding to the
$\Bin(n,p)$ distribution normalized so that $\E_{\Bin(n,p)} K_i^2(Z)=1$. Write
\begin{equation*}
R(z)=\sum_{i=0}^{k/2} a_iK_i(z).
\end{equation*}
Note that
\begin{equation*}
\E_{\Bin(n,p)} P(Z) = \E_{\Bin(n,p)} R^2(Z) = \sum_{i=0}^{k/2} a_i^2.
\end{equation*}
Hence the problem \eqref{one_d_max_dual_continuous_problem_sq_form} reduces to minimizing
 $\sum_{i=0}^{k/2} a_i^2$ under the constraint that $|R(n)| = |\sum_{i=0}^{k/2} a_iK_i(n)| \ge 1$.
 By Cauchy-Schwarz,
\begin{equation*}
\sum_{i=0}^{k/2} a_i^2
\sum_{i=0}^{k/2}K_i^2(n) \ge
\Big(\sum_{i=0}^{k/2} a_i K_i(n)\Big)^2 \ge 1.
\end{equation*}
Hence the optimal value of the problem
\eqref{one_d_max_dual_continuous_problem_sq_form} is
$\frac{1}{\sum_{i=0}^{k/2} K_i^2(n)}$ and the optimal polynomial
is (up to multiplication by $(-1)$)
\begin{equation*}
R(z)=\frac{1}{\sum_{i=0}^{k/2} K_i^2(n)}\sum_{i=0}^{k/2} K_i(n)K_i(z).
\end{equation*}
Since the Krawtchouk polynomials equal \cite{Sz75}
\begin{equation*}
K_i(x)=\Choose{n}{i}^{-\frac{1}{2}}\left(p(1-p)\right)^{-\frac{i}{2}}\sum_{j=0}^i
(-1)^{i-j} \Choose{n-x}{i-j}\Choose{x}{j}p^{i-j}(1-p)^{j},
\end{equation*}
and in particular
\begin{equation*}
K_i(n)=\Choose{n}{i}^{\frac{1}{2}}\left(\frac{1-p}{p}\right)^{\frac{i}{2}},
\end{equation*}
we deduce \eqref{value_of_tilde_M}.
\end{proof}
\end{section}
\end{document}